\documentclass[11pt]{amsart}
\usepackage[T1]{fontenc}

\usepackage{amssymb}
\usepackage{enumerate}
\usepackage{hyperref}
\usepackage{color}
\usepackage{graphicx}
\usepackage{amsmath}
\usepackage{mathtools}

\input xy
\xyoption{all}

\newcommand{\N}{\mathbb{N}}

\renewcommand{\phi}{\varphi}

\newtheorem{theorem}{Theorem}[section]

\newtheorem{lemma}[theorem]{Lemma}

\newtheorem{question}[theorem]{Question}
\newtheorem{corollary}[theorem]{Corollary}
\theoremstyle{definition}

\newtheorem{remark}[theorem]{Remark}

\usepackage{easyReview}

\title{All iterated function systems are Lipschitz up to an equivalent metric}

\author{ Micha\l{} Pop\l{}awski}

\address{M. Pop\l{}awski
	(ORCID 0000-0002-2725-9675): 
	Mathematics Department, 
	Jan Kochanowski University in Kielce,
	Uniwersytecka 7, 25-406 Kielce, Poland}
\email{michal.poplawski.m@gmail.com}

\date{\today}

\begin{document}
	
	\begin{abstract} A finite family $\mathcal{F}=\{f_1,\ldots,f_n\}$ of continuous selfmaps of a given metric space $X$ is called an iterated function system (shortly IFS). In a case of contractive selfmaps of a complete metric space is well-known that IFS has an unique attractor \cite{Hu}. However, in \cite{LS} authors studied highly non-contractive IFSs, i.e. such families $\mathcal{F}=\{f_1,\ldots,f_n\}$ of continuous selfmaps that for any remetrization of $X$ each function $f_i$ has Lipschitz constant $>1, i=1,\ldots,n.$ They asked when one can remetrize $X$ that $\mathcal{F}$ is Lipschitz IFS, i.e. all $f_i's$ are Lipschitz (not necessarily contractive), $ i=1,\ldots,n$. We give a general positive answer for this problem by constructing respective new metric (equivalent to the original one) on $X$, determined by a given family $\mathcal{F}=\{f_1,\ldots,f_n\}$ of continuous selfmaps of $X$. However, our construction is valid even for some specific infinite families of continuous functions.
	\end{abstract}
	\maketitle
	%
	
	\section{Introduction}
	We will start with brief discussion on well-known notions of metric spaces and IFS theories. Suppose that $(X,d)$ is a metric space.  For any selfmap $f \colon X \to X$ we define the Lipschitz constant $L_d(f)$ of $f$ as follows: 
	$$L_d(f)=\sup\left\{\frac{d(f(x),f(y))}{d(x,y)} \colon x,y \in X, \ x \neq y\right\}.$$ If $L_d(f)<\infty$ we say that $f$ is Lipschitz. In the case $L_d(f)<1$ we say that $f$ is contraction. 
	Denote the family of all compact subsets of $X$ by $K(X).$ The Hausdorff metric $d_H$ on $K(X) \setminus \{\emptyset\}$ is given by
	$$d_H(A,B)=\max\{\sup\{\inf\{d(a,x) \colon x \in B\} \colon a \in A\},\sup\{\inf\{d(y,b) \colon y \in A\} \colon b \in B\}\}$$ for any nonempty $A,B \in K(X).$ 
	
	Any family $\mathcal{F}=\{f_1,\ldots,f_n\}$ consisting of continuous selfmaps of $X$ is called an iterated function system, shortly IFS. The Hutchinson operator $F$ determined by an IFS $\mathcal{F}=\{f_1,\ldots,f_n\}$ is defined as a function $F \colon K(X) \to K(X)$ given by $F[A]=\bigcup_{i=1}^n f_i[A]$ for any $A \in K(X).$ 
	We say that $A \in K(X)$ is an attractor of $\mathcal{F}$ if $F(A)=A$ and $F(K) \to A$ with respect to Hausdorff metric $d_H.$
	Hutchinson proved in \cite{Hu} that in the case of complete metric space $X$ and IFS $\mathcal{F}$ consisting of contractive selfmaps
	there is a unique attractor of $\mathcal{F}.$
	Above result provokes to investigate the $L-$expansive \cite{LS}  IFSs, i.e. IFSs $\mathcal{F}$ satisfying $L_{\rho}(f)>1$ for any $f \in \mathcal{F}$ and a metric $\rho$ equivalent to original metric $d$. In \cite{LS} authors constructed $L-$expansive IFSs with attractors.  These examples are based on the non-Lipschitz square root function $ [0,1] \ni x \mapsto \sqrt{x}$ somehow. The following problem arises 
	\begin{question} \cite[Question 2.]{LS}
		Let $\mathcal{F}$ be an IFS on a metric space $(X,d).$ When exists a metric $\rho$ equivalent to $d$ such that all $f \in \mathcal{F}$ are Lipschitz?
	\end{question}
	
	A partial asnwer was given in the same paper.
	
	\begin{theorem}\cite[Theorem 23.]{LS}
	Consider the unit interval $[0,1]$ with the euclidean metric. Let $f \colon [0,1] \to [0,1]$ be a strictly increasing homeomorphic map with $f(0)=0$. Suppose that there is $a \in (0,1)$ such that $f(x)>x$ for all $x \in (0,a]$ and $f$ is continuously differentiable on $[f^{-1}(a),1].$ Then there is an equivalent metric on $[0,1]$ such that $f$ is Lipschitz.
	\end{theorem}

	\begin{remark}
		Note than we can not expect a general method of making any continuous selfmap a contractive one. For instance the Banach fixed point theorem forces an existence of a fixed point in the case of a contractive self map of a complete metric space $X$.
	\end{remark}

	\section{Equicontinuity}
	Now, let us introduce specific notations and commonly known results on equicontinuity.
	Suppose $(X,d)$ is a metric space and fix a family $\mathcal{F} \subset C(X)=\{f \colon X \to X \colon f \textrm{ is continuous}\}.$ 
	We say that  $\mathcal{F}$ is equicontinuous at a point $x \in X$ if for any $\varepsilon>0$ there is a $\delta>0$ such that 
	$$\forall_{y \in X} \ d(x,y)<\delta \Rightarrow \forall_{f \in \mathcal{F}} \  d(f(x),f(y))<\varepsilon.$$ If $\mathcal{F}$ is equicontinuous at all points $x \in X$, then we say that $\mathcal{F}$ is pointwise equicontinuous.
	We say that $\mathcal{F}$ is uniformly equicontinuous if for any $\varepsilon>0$ there is a $\delta>0$ such that 
	$$\forall_{x,y \in X} \ d(x,y)<\delta \Rightarrow \forall_{f \in \mathcal{F}} \ d(f(x),f(y))<\varepsilon.$$
	Obviously, pointwise equicontinuity follows from uniform equicontinuity.
	
	If $\mathcal{F},\mathcal{G} \subset C(X)$ we denote $\mathcal{F} \circ \mathcal{G}=\{f \circ g \colon f \in \mathcal{F}, \ g \in \mathcal{G}\}$ and $\mathcal{F}^n=\{f_1 \circ \ldots \circ f_n \colon f_1,\ldots,f_n \in \mathcal{F}\}.$ 
	We are interested in the following property 
	$$\forall_{n \in \N} \ \mathcal{F}^n \textrm{ is pointwise equicontinuous.} $$
	The following results should be known.
	\begin{lemma}
		Let $\mathcal{F},\mathcal{G} \subset C(X)$ be uniformly equicontinuous. Then $\mathcal{F} \circ \mathcal{G}$ is uniformly equicontinuous, too. 
	\end{lemma}
	\begin{corollary} \label{upot}
		Let $\mathcal{F} \subset C(X)$ be uniformly continuous. Then for any $n \in \N$ the family $\mathcal{F}^n$ is uniformly continuous.
	\end{corollary}
	\begin{lemma} \label{psum}
		Let $\mathcal{F},\mathcal{G} \subset C(X)$ be pointwise equicontinuous. Then $\mathcal{F} \cup \mathcal{G}$ is pointwise equicontinuous, too.
	\end{lemma}
	\begin{lemma} \label{ifsa}
		Any finite family $\mathcal{F} \subset C(X)$ is pointwise equicontinuous. Especially, for any $n \in \N,$ IFS $\mathcal{F}$ any family $\mathcal{F}^n$ is pointwise equicontinuous.
	\end{lemma}
	Using Corollary \ref{upot} one can deduce the following result.
	\begin{lemma} \label{piu}
		For any compact metric space $(X,d)$ and pointwise equicontinuous family $\mathcal{F} \subset C(X)$ all families $\mathcal{F}^n$ are uniformly continuous, $n \in \N.$
	\end{lemma}
	
	\section{Main result}
	\begin{theorem} \label{pct}
		Suppose $(X,d)$ is a metric space and $\mathcal{F} \subset C(X)$ is a family such that $\mathcal{F}^n$ is pointwise equicontinuous for each $n \in \N$. Then for any $\varepsilon>0$ there is an equivalent metric $d_{\varepsilon}$ such that $L_{d_{\varepsilon}}(f) \leq 1+\varepsilon$ for any $f \in \mathcal{F}.$ 
	\end{theorem}
	\begin{proof}
		We may suppose that $d \leq 1$ (if $d>1$ put $d:=\frac{d}{d+1}$.) Define $$\rho(x,y)=d_{\varepsilon}(x,y)=d(x,y)+\sum_{n=1}^{\infty} \frac{\sup\{d(f(x),f(y)) \colon f \in \mathcal{F}^n\}}{(1+\varepsilon)^n}.$$
		It is easy to see that $\rho$ is metric and $\rho \geq d.$ 
		
		In order to proof that $\rho$ and $d$ are equivalent we fix a sequence $(x_k)_{k \in \N}$ with $x_k \to x$ as $k \to \infty$ in $(X,d)$ for some $x \in X.$ We will check that $\rho(x_k,x) \to 0$ as $k \to \infty.$
		 
		 Fix $\eta>0$ and find $k_1 \in \N$ with $d(x_k,x)<\frac{\eta}{3}$ for any $k \geq k_1.$
		 We find $N \in \N$ large enough to ensure $\sum_{n=N+1}^{\infty} \frac{1}{(1+\varepsilon)^n}<\frac{\eta}{3}.$ By our assumption, families $\mathcal{F}^n, \ n \leq N$ are pointwise equicontinuous and consequently by Lemma \ref{psum}, the union $\bigcup_{n=1}^{N} \mathcal{F}^N$ is pointwise equicontinuous, too. Then, we may choose $\delta>0$ such that
		 $$\forall_{y \in X} \ d(x,y)<\delta \Rightarrow \forall_{f \in \mathcal{F}^n, \ n \leq N} \  d(f(x),f(y))<\frac{\eta \cdot \varepsilon}{3}.$$
		 Find $k_2 \geq k_1$ such that $d(x_k,x)<\delta$ for any $k \geq k_2.$ As a result, for any $k \geq k_2$ we have
		 $$\rho(x_k,x)=d(x_k,x)+\sum_{n=1}^{N} \frac{\sup\{d(f(x_k),f(x)) \colon f \in \mathcal{F}^n\}}{(1+\varepsilon)^n}+\sum_{n=N+1}^{\infty} \frac{\sup\{d(f(x_k),f(x)) \colon f \in \mathcal{F}^n\}}{(1+\varepsilon)^n}$$
		 $$<\frac{\eta}{3}+\sum_{n=1}^{N} \frac{\frac{\eta \cdot \varepsilon}{3}}{(1+\varepsilon)^n}+\frac{\eta}{3}<\frac{2\eta}{3}+\frac{\eta \varepsilon }{3} \sum_{n=1}^{\infty} \frac{1}{(1+\varepsilon)^n}=\eta.$$
		 
		 Now, we will check that any $g \in \mathcal{F}$ is $(1+\varepsilon)-$Lipschitz as $g \colon (X,\rho) \to (X,\rho).$
		 
		 Fix $x,y \in X, \ g \in \mathcal{F}$. Firstly, note that \begin{equation} \label{ps} d(g(x),g(y))=(1+\varepsilon) \cdot \frac{d(g(x),g(y))}{1+\varepsilon}. \end{equation}
		 Secondly, for any $n \in \N, \ h \in \mathcal{F}^n$ we have $h \circ g \in \mathcal{F}^{n+1}$ and consequently
		 $$\frac{d(h(g(x)),h(g(y)))}{(1+\varepsilon)^{n}} \leq (1+\varepsilon) \cdot \frac{\sup\{d(f(x),f(y)) \colon f \in \mathcal{F}^{n+1}\}}{(1+\varepsilon)^{n+1}}.$$ Taking supremum over $h \in \mathcal{F}^n$ we get
		 $$\frac{\sup\{d(h(g(x)),h(g(y))) \colon h \in \mathcal{F}^n\}}{(1+\varepsilon)^{n}} \leq (1+\varepsilon) \cdot \frac{\sup\{d(f(x),f(y)) \colon f \in \mathcal{F}^{n+1}\}}{(1+\varepsilon)^{n+1}}.$$ Summing up these inequalities over $n \in \N$ we have
		 \begin{align}\label{rs}
		 &\sum_{n=1}^{\infty} 
		 \frac{\sup\{d(f(g(x)),f(g(y))) \colon f \in \mathcal{F}^n\}}{(1+\varepsilon)^{n}} \leq (1+\varepsilon) \cdot \sum_{n=1}^{\infty} \frac{\sup\{d(f(x),f(y)) \colon f \in \mathcal{F}^{n+1}\}}{(1+\varepsilon)^{n+1}}		
		 \\
		 \label{rcd}
		 &=(1+\varepsilon) \cdot \sum_{n=2}^{\infty} \frac{\sup\{d(f(x),f(y)) \colon f \in \mathcal{F}^{n}\}}{(1+\varepsilon)^{n}}.
		 \end{align}
		 Using inequalities (\ref{ps}),(\ref{rs})  (\ref{rcd}) we have
		 $$\rho(g(x),g(y)) \leq (1+\varepsilon) \cdot \frac{d(g(x),g(y))}{1+\varepsilon}+(1+\varepsilon) \cdot \sum_{n=2}^{\infty} \frac{\sup\{d(f(x),f(y)) \colon f \in \mathcal{F}^{n}\}}{(1+\varepsilon)^{n}}
		 $$
		 $$\leq (1+\varepsilon) \cdot \sum_{n=1}^{\infty} \frac{\sup\{d(f(x),f(y)) \colon f \in \mathcal{F}^{n}\}}{(1+\varepsilon)^{n}} \leq (1+\varepsilon) \cdot \rho(x,y).$$
	\end{proof}
	
	\section{Further remarks and conclusions}
	As a particular application of Theorem \ref{pct} and its proof we get the result for one continuous function, which extends \cite[Theorem 23.]{LS}.
	
	\begin{theorem} \label{oft}
		Let $(X,d)$ be a metric space and $f \colon (X,d) \to (X,d)$ be a continuous map. For any $\varepsilon>0$ there is an equivalent metric $d_{f,\epsilon}$ on $X$  given by . 
		$$d_{f,\varepsilon}(x,y)=d(x,y)+\sum_{n=1}^{\infty}\frac{d(f^n(x),f^n(y))}{(1+\varepsilon)^n}$$ 	with the notation: $f^n$ is $n-$th autocomposition of $f$ (instead of $n-$th power).
		Then $f$ is $(1+\varepsilon)-$Lipschitz as a map $f \colon (X,d_{f,\varepsilon}) \to  (X,d_{f,\varepsilon}).$
	\end{theorem}
	
	Furthermore, by Lemma \ref{ifsa} any IFS satisfies assumptions of Theorem \ref{pct}. Then, by Theorem \ref{pct} we answer to the \cite[Question 2.]{LS}.
	\begin{theorem}
		Suppose $(X,d)$ is a metric space and $\mathcal{F} \subset C(X)$ is an IFS. Then for any $\varepsilon>0$ there is an equivalent metric $d_{\varepsilon}$ such that any $\mathcal{F} \ni f \colon (X,d_{\varepsilon}) \to (X,d_{\varepsilon})$ is $(1+\varepsilon)-$Lipschitz.
	\end{theorem}
	
	One can use Lemma \ref{piu} to get another observation.
	
	\begin{corollary}
		Suppose that $\mathcal{F} \subset C(X)$ is a pointwise equicontinuous family of functions defined on a compact metric space $(X,d).$ Then for any $\varepsilon>0$ there is an equivalent metric $d_{\varepsilon}$ such that any $\mathcal{F} \ni f \colon (X,d_{\varepsilon}) \to (X,d_{\varepsilon})$ is $(1+\varepsilon)-$Lipschitz.
	\end{corollary}
	
	Since a substitution $d:=\frac{d}{d+1}$ preserves completeness of $d$ and metric $\rho$ constructed in Theorem \ref{pct} satisfies $\rho \geq d$ then $\rho$ preserves completeness, too. In other words:
	
	\begin{remark}
		If a metric space $(X,d)$ is complete, then a metric $\rho$ defined in Theorem \ref{pct} is complete.
	\end{remark}
	
	Now, recall a notion of generalized joint spectral radius (see \cite{JJ}). Suppose that $\mathcal{F} \subset C(X)$ for some metric space $(X,d).$ 
	The limit $r_d(\mathcal{F}):=\lim_{n \to \infty} (L_d(\mathcal{F}^n))^{\frac{1}{n}}$ is called generalized joint spectral radius of $\mathcal{F}.$
	Recall that metrics $d,\rho$ defined on a set $X$ are Lipschitz equivalent if both identities $id_{d,\rho} \colon (X,d) \to (X,\rho), \ id_{\rho,d} \colon (X,\rho) \to (X,d)$ are Lipschitz.
	In \cite{JJ} the following result is obtained:
	\begin{theorem}\cite[Corollary 3.4]{JJ}
		Let $(X,d)$ be a metric space. Suppose that $\mathcal{F} \subset C(X)$ satisfies $L_d(\mathcal{F})<\infty.$ Then $r_d(\mathcal{F})<1$ if and only if there is a metric $\rho$ Lipschitz equivalent $d$ with $L_{\rho}(\mathcal{F})<1.$ 
	\end{theorem}
	
	\begin{remark}
		A metric $\rho$ constructed in Theorem \ref{pct} for a given $\varepsilon>0$ satisfies: $r_{\rho}(\mathcal{F}) \leq 1+\varepsilon.$
	\end{remark}
	\begin{proof}
		Fix $n \in \N$ and $g \in \mathcal{F}^n$. Similarly as in proof of Theorem \ref{pct} we have
		$$d(g(x),g(y))=(1+\varepsilon)^n \cdot \frac{d(g(x),g(y))}{(1+\varepsilon)^n}$$ and
		$$\sum_{k=1}^{\infty} 
		\frac{\sup\{d(f(g(x)),f(g(y))) \colon f \in \mathcal{F}^k\}}{(1+\varepsilon)^{k}} \leq (1+\varepsilon)^n \cdot \sum_{k=n+1}^{\infty} \frac{\sup\{d(f(x),f(y)) \colon f \in \mathcal{F}^{k}\}}{(1+\varepsilon)^{k}}$$ (compare with inequalities (\ref{ps}),(\ref{rs}),(\ref{rcd}).) Consequently, $L_{\rho}(g) \leq (1+\varepsilon)^n, \ L_{\rho}(\mathcal{F}^n) \leq (1+\varepsilon)^n$ and $r_{\rho}(\mathcal{F}) \leq \lim_{n \to \infty} ((1+\varepsilon)^n)^{\frac{1}{n}}=1+\varepsilon.$
	\end{proof}
	The above result underlies relevance of distinction between two cases: $r_d(\mathcal{F})<1$ and $r_d(\mathcal{F}) \geq 1.$

	\end{document}